\documentclass{amsart}
\usepackage{amssymb}
\usepackage{amsmath}
\usepackage{amsfonts}
\usepackage{hyperref}

\newtheorem{theorem}{Theorem}
\theoremstyle{plain}

\newtheorem{corollary}{Corollary}

\newtheorem{example}{Example}

\newtheorem{lemma}{Lemma}

\newtheorem{problem}{Problem}
\newtheorem{proposition}{Proposition}

\numberwithin{equation}{section}
\input{tcilatex}

\begin{document}
\title[Asymptotic formulas for the gamma function]{Asymptotic formulas for
the gamma function constructed by bivariate means}
\author{Zhen-Hang Yang}
\address{Power Supply Service Center, ZPEPC Electric Power Research
Institute, Hangzhou, Zhejiang, China, 310007}
\email{yzhkm@163.com}
\date{July 19, 2014}
\subjclass[2010]{Primary 33B15, 26E60; Secondary 26D15, 11B83}
\keywords{Stirling's formula, gamma function, mean, inqueality, polygamma
function}
\thanks{This paper is in final form and no version of it will be submitted
for publication elsewhere.}

\begin{abstract}
Let $K,M,N$ denote three bivariate means. In the paper, the author prove the
asymptotic formulas for the gamma function have the form of%
\begin{equation*}
\Gamma \left( x+1\right) \thicksim \sqrt{2\pi }M\left( x+\theta ,x+1-\theta
\right) ^{K\left( x+\epsilon ,x+1-\epsilon \right) }e^{-N\left( x+\sigma
,x+1-\sigma \right) }
\end{equation*}%
or%
\begin{equation*}
\Gamma \left( x+1\right) \thicksim \sqrt{2\pi }M\left( x+\theta ,x+\sigma
\right) ^{K\left( x+\epsilon ,x+1-\epsilon \right) }e^{-M\left( x+\theta
,x+\sigma \right) }
\end{equation*}%
as $x\rightarrow \infty $, where $\epsilon ,\theta ,\sigma $ are fixed real
numbers. This idea can be extended to the psi and polygamma functions. As
examples, some new asymptotic formulas for the gamma function are presented.
\end{abstract}

\maketitle

\section{Introduction}

The Stirling's formula%
\begin{equation}
n!\thicksim \sqrt{2\pi n}n^{n}e^{-n}:=s_{n}  \label{S}
\end{equation}%
has important applications in statistical physics, probability theory and
and number theory. Due to its practical importance, it has attracted much
interest of many mathematicians and have motivated a large number of
research papers concerning various generalizations and improvements.

Burnside's formula \cite{Burnside-MM-46-1917}%
\begin{equation}
n!\thicksim \sqrt{2\pi }\left( \frac{n+1/2}{e}\right) ^{n+1/2}:=b_{n}
\label{B}
\end{equation}%
slight improves (\ref{S}). Gosper \cite{Gosper-PNAS-75-1978} replaced $\sqrt{%
2\pi n}$ by $\sqrt{2\pi \left( n+1/6\right) }$ in (\ref{S}) to get%
\begin{equation}
n!\thicksim \sqrt{2\pi \left( n+\tfrac{1}{6}\right) }\left( \frac{n}{e}%
\right) ^{n}:=g_{n},  \label{G}
\end{equation}%
which is better than (\ref{S}) and (\ref{B}). In the recent paper \cite%
{Batir-P-27(1)-2008}, N. Batir obtained an asymptotic formula similar to (%
\ref{G}): 
\begin{equation}
n!\thicksim \frac{n^{n+1}e^{-n}\sqrt{2\pi }}{\sqrt{n-1/6}}:=b_{n}^{\prime },
\label{Batir1}
\end{equation}%
which is stronger than (\ref{S}) and (\ref{B}). A more accurate
approximation for the factorial function%
\begin{equation}
n!\thicksim \sqrt{2\pi }\left( \frac{n^{2}+n+1/6}{e^{2}}\right)
^{n/2+1/4}:=m_{n}  \label{M}
\end{equation}%
was presented in \cite{Mortici-CMI-19(1)-2010} by Mortici.

The classical Euler's gamma function $\Gamma $ may be defined by%
\begin{equation}
\Gamma \left( x\right) =\int_{0}^{\infty }t^{x-1}e^{-t}dt  \label{Gamma}
\end{equation}%
for $x>0$, and its logarithmic derivative $\psi \left( x\right) =\Gamma
^{\prime }\left( x\right) /\Gamma \left( x\right) $ is known as the psi or
digamma function, while $\psi ^{\prime }$, $\psi ^{\prime \prime }$, ... are
called polygamma functions (see \cite{Anderson-PAMS-125(11)-1997}).

The gamma function is closely related to the Stirling's formula, since $%
\Gamma (n+1)=n!$ for all $n\in \mathbb{N}$. This inspires some authors to
also pay attention to find better approximations for the gamma function. For
example, Ramanujan's \cite[P. 339]{Ramanujan-SB-1988} double inequality for
the gamma function:%
\begin{equation}
\sqrt{\pi }\left( \tfrac{x}{e}\right) ^{x}\left( 8x^{3}+4x^{2}+x+\tfrac{1}{%
100}\right) ^{1/6}<\Gamma \left( x+1\right) <\sqrt{\pi }\left( \tfrac{x}{e}%
\right) ^{x}\left( 8x^{3}+4x^{2}+x+\tfrac{1}{30}\right) ^{1/6}  \label{R}
\end{equation}%
for $x\geq 1$. Batir \cite{Batir-AM-91-2008} showed that for $x>0$,%
\begin{eqnarray}
&&\sqrt{2}e^{4/9}\left( \frac{x}{e}\right) ^{x}\sqrt{x+\frac{1}{2}}\exp
\left( -\tfrac{1}{6\left( x+3/8\right) }\right)  \label{Batir2} \\
&<&\Gamma \left( x+1\right) <\sqrt{2\pi }\left( \frac{x}{e}\right) ^{x}\sqrt{%
x+\frac{1}{2}}\exp \left( -\tfrac{1}{6\left( x+3/8\right) }\right) .  \notag
\end{eqnarray}%
Mortici \cite{Mortici-AM-93-2009-1} proved that for $x\geq 0$, 
\begin{eqnarray}
\sqrt{2\pi e}e^{-\omega }\left( \frac{x+\omega }{e}\right) ^{x+1/2}
&<&\Gamma \left( x+1\right) \leq \alpha \sqrt{2\pi e}e^{-\omega }\left( 
\frac{x+\omega }{e}\right) ^{x+1/2},  \label{Ml} \\
\beta \sqrt{2\pi e}e^{-\varsigma }\left( \frac{x+\varsigma }{e}\right)
^{x+1/2} &<&\Gamma \left( x+1\right) \leq \sqrt{2\pi e}e^{-\varsigma }\left( 
\frac{x+\varsigma }{e}\right) ^{x+1/2}  \label{Mr}
\end{eqnarray}%
where $\omega =\left( 3-\sqrt{3}\right) /6$, $\alpha =1.072042464...$ and $%
\varsigma =\left( 3+\sqrt{3}\right) /6$, $\beta =0.988503589...$.

More results involving the asymptotic formulas for the factorial or gamma
functions can consult \cite{Shi-JCAM-195-2006}, \cite{Guo-JIPAM-9(1)-2008}, 
\cite{Mortici-MMN-11(1)-2010}, \cite{Mortici-CMA-61-2011}, \cite%
{Zhao-PMD-80(3-4)-2012}, \cite{Mortici-MCM-57-2013}, \cite{Qi-JCAM-268-2014}%
, \cite{Qi-JCAM-268-2014}, \cite{Lu-RJ-35(1)-2014} and the references cited
therein).

Mortici \cite{Mortici-BTUB-iii-3(52)-2010} presented an idea that by
replacing an under-approximation and an upper-approximation of the factorial
function by one of their geometric mean to improve certain approximation
formula of the factorial. In fact, by observing and analyzing these
asymptotic formulas for factorial or gamma function, we find out that they
have the common form of%
\begin{equation}
\ln \Gamma \left( x+1\right) \thicksim \frac{1}{2}\ln 2\pi +P_{1}\left(
x\right) \ln P_{2}\left( x\right) -P_{3}\left( x\right) +P_{4}\left(
x\right) ,  \label{g-form}
\end{equation}%
where $P_{1}\left( x\right) ,P_{2}\left( x\right) $ and $P_{3}\left(
x\right) $ are all means of $x$ and $\left( x+1\right) $, while $P_{4}\left(
x\right) $ satisfies $P_{4}\left( \infty \right) =0$. For example, (\ref{S}%
)--(\ref{M}) can be written as%
\begin{eqnarray*}
&&\ln n!\thicksim \frac{1}{2}\ln 2\pi +\left( n+\frac{1}{2}\right) \ln n-n,
\\
&&\ln n!\thicksim \frac{1}{2}\ln 2\pi +\left( n+\frac{1}{2}\right) \ln
\left( n+\frac{1}{2}\right) -\left( n+\frac{1}{2}\right) , \\
&&\ln n!\thicksim \frac{1}{2}\ln 2\pi +\left( n+\frac{1}{2}\right) \ln n-n+%
\frac{1}{2}\ln \left( 1+\tfrac{1}{6n}\right) , \\
&&\ln n!\thicksim \frac{1}{2}\ln 2\pi +\left( n+\frac{1}{2}\right) \ln n-n-%
\frac{1}{2}\ln \left( 1-\tfrac{1}{6n}\right) , \\
&&\ln n!\thicksim \frac{1}{2}\ln 2\pi +\left( n+\frac{1}{2}\right) \ln \sqrt{%
\frac{n^{2}+4n\left( n+1\right) +\left( n+1\right) ^{2}}{6}}-\left( n+\frac{1%
}{2}\right) .
\end{eqnarray*}%
Inequalities (\ref{R})--(\ref{Mr}) imply that%
\begin{eqnarray*}
&&\ln \Gamma \left( x+1\right) \thicksim \frac{1}{2}\ln 2\pi +\left( x+\frac{%
1}{2}\right) \ln x-x+\frac{1}{6}\ln \left( 1+\frac{1}{2x}+\frac{1}{8x^{2}}+%
\frac{1}{240x^{3}}\right) , \\
&&\ln \Gamma \left( x+1\right) \thicksim \frac{1}{2}\ln 2\pi +\left( x+\frac{%
1}{2}\right) \ln x-x+\frac{1}{2}\ln \left( 1+\frac{1}{2x}\right) -\tfrac{1}{%
6\left( x+3/8\right) }, \\
&&\ln \Gamma \left( x+1\right) \thicksim \frac{1}{2}\ln 2\pi +\left( x+\frac{%
1}{2}\right) \ln \left( \left( 1-a\right) x+a\left( x+1\right) \right)
-\left( \left( 1-a\right) x+a\left( x+1\right) \right) ,
\end{eqnarray*}%
where $a=\omega =(3-\sqrt{3})/6$, $\varsigma =(3+\sqrt{3})/6$.

The aim of this paper is to prove the validity of the form (\ref{g-form})
which offers such a new way to construct asymptotic formulas for Euler gamma
function in terms of bivariate means. Our main results are included in
Section 2. Some new examples are presented in the last section.

\section{Main results}

Before stating and proving our main results, we recall some knowledge on
means. Let $I$ be an interval on $\mathbb{R}$. A bivariate real valued
function $M:I^{2}\rightarrow \mathbb{R}$ is said to be a bivariate mean if 
\begin{equation*}
\min \left( a,b\right) \leq M\left( a,b\right) \leq \max \left( a,b\right)
\end{equation*}%
for all $a,b\in I$. Clearly, each bivariate mean $M$ is reflexive, that is, 
\begin{equation*}
M\left( a,a\right) =a
\end{equation*}%
for any $a\in I$. $M$ is symmetric if 
\begin{equation*}
M\left( a,b\right) =M\left( b,a\right)
\end{equation*}%
for all $a,b\in I$, and $M$ is said to be homogeneous (of degree one) if 
\begin{equation}
M\left( ta,tb\right) =tM\left( a,b\right)  \label{M-h}
\end{equation}%
for any $a,b\in I$ and $t>0$.

The lemma is crucial to prove our results.

\begin{lemma}[{\protect\cite[Thoerem 1, 2, 3]{Toader.MIA.5.2002}}]
\label{Lemma M}If $M:I^{2}\rightarrow \mathbb{R}$ is a differentiable mean,
then for $c\in I$,%
\begin{equation*}
M_{a}^{\prime }\left( c,c\right) ,M_{b}^{\prime }\left( c,c\right) \in
\left( 0,1\right) \text{ \ and \ }M_{a}^{\prime }\left( c,c\right)
+M_{b}^{\prime }\left( c,c\right) =1\text{.}
\end{equation*}%
In particular, if $M$ is symmetric, then%
\begin{equation*}
M_{a}^{\prime }\left( c,c\right) =M_{b}^{\prime }\left( c,c\right) =1/2.
\end{equation*}
\end{lemma}

Now we are in a position to state and prove main results.

\begin{theorem}
\label{MT-p2><p3}Let $M:\left( 0,\infty \right) \times \left( 0,\infty
\right) \rightarrow \left( 0,\infty \right) $ and $N:\left( -\infty ,\infty
\right) \times \left( -\infty ,\infty \right) \rightarrow \left( -\infty
,\infty \right) $ be two symmetric, homogeneous and differentiable means and
let $r$ be defined on $\left( 0,\infty \right) $ satisfying $%
\lim_{x\rightarrow \infty }r\left( x\right) =0$. Then for fixed real numbers 
$\theta ,\theta ^{\ast },\sigma ,\sigma ^{\ast }$ with $\theta +\theta
^{\ast }=\sigma +\sigma ^{\ast }=1$ such that $x>-\min \left( 1,\theta
,\theta ^{\ast }\right) $, we have%
\begin{equation*}
\Gamma \left( x+1\right) \thicksim \sqrt{2\pi }M\left( x+\theta ,x+\theta
^{\ast }\right) ^{x+1/2}e^{-N\left( x+\sigma ,x+\sigma ^{\ast }\right)
}e^{r\left( x\right) }\text{, as }x\rightarrow \infty .
\end{equation*}
\end{theorem}

\begin{proof}
Since $\lim_{x\rightarrow \infty }r\left( x\right) =0$, the desired result
is equivalent to%
\begin{equation*}
\lim_{x\rightarrow \infty }\left( \ln \Gamma \left( x+1\right) -\ln \sqrt{%
2\pi }-\left( x+\frac{1}{2}\right) \ln M\left( x+\theta ,x+\theta ^{\ast
}\right) +N\left( x+\sigma ,x+\sigma ^{\ast }\right) \right) =0.
\end{equation*}%
Due to $\lim_{x\rightarrow \infty }r\left( x\right) =0$ and the known
relation%
\begin{equation*}
\lim_{x\rightarrow \infty }\left( \ln \Gamma \left( x+1\right) -\left( x+%
\frac{1}{2}\right) \ln \left( x+\frac{1}{2}\right) +\left( x+\frac{1}{2}%
\right) \right) =\frac{1}{2}\ln 2\pi ,
\end{equation*}%
it suffices to prove that 
\begin{eqnarray*}
D_{1} &:&=\lim_{x\rightarrow \infty }\left( x+\frac{1}{2}\right) \ln \frac{%
M\left( x+\theta ,x+\theta ^{\ast }\right) }{x+1/2}=0, \\
D_{2} &:&=\lim_{x\rightarrow \infty }\left( N\left( x+\sigma ,x+\sigma
^{\ast }\right) -\left( x+\frac{1}{2}\right) \right) =0.
\end{eqnarray*}

Letting $x=1/t$, using the homogeneity of $M$, that is, (\ref{M-h}), and
utilizing L'Hospital rule give%
\begin{eqnarray*}
D_{1} &=&\lim_{t\rightarrow 0^{+}}\frac{1+t/2}{t}\ln \frac{M\left( 1+\theta
t,1+\theta ^{\ast }t\right) }{1+t/2} \\
&=&\lim_{t\rightarrow 0^{+}}\frac{\ln M\left( 1+\theta t,1+\theta ^{\ast
}t\right) -\ln \left( 1+t/2\right) }{t} \\
&=&\lim_{t\rightarrow 0^{+}}\left( \frac{\theta M_{x}\left( 1+\theta
t,1+\theta ^{\ast }t\right) +\theta ^{\ast }M_{y}\left( 1+\theta t,1+\theta
^{\ast }t\right) }{M\left( 1+\theta t,1+\theta ^{\ast }t\right) }-\frac{1}{%
2+t}\right) \\
&=&\frac{\theta M_{x}\left( 1,1\right) +\theta ^{\ast }M_{y}\left(
1,1\right) }{M\left( 1,1\right) }-\frac{1}{2}=0,
\end{eqnarray*}%
where the last equality holds due to Lemma \ref{Lemma M}.

Similarly, we have 
\begin{eqnarray*}
D_{2} &=&\lim_{x\rightarrow \infty }\left( N\left( x+\sigma ,x+\sigma ^{\ast
}\right) -\left( x+\frac{1}{2}\right) \right) \\
&&\overset{1/x=t}{=\!=\!=}\lim_{t\rightarrow 0^{+}}\frac{N\left( 1+\sigma
t,1+\sigma ^{\ast }t\right) -\left( 1+t/2\right) }{t} \\
&=&\lim_{t\rightarrow 0^{+}}\left( \sigma N_{x}\left( 1+\sigma t,1+\sigma
^{\ast }t\right) +\sigma ^{\ast }N_{y}\left( 1+\sigma t,1+\sigma ^{\ast
}t\right) -\frac{1}{2}\right) \\
&=&\frac{\sigma +\sigma ^{\ast }}{2}-\frac{1}{2}=0,
\end{eqnarray*}%
which proves the desired result.
\end{proof}

\begin{theorem}
\label{MT-p2=p3}Let $M:\left( 0,\infty \right) \times \left( 0,\infty
\right) \rightarrow \left( 0,\infty \right) $ be a mean and let $r$ be
defined on $\left( 0,\infty \right) $ satisfying $\lim_{x\rightarrow \infty
}r\left( x\right) =0$. Then for fixed real numbers $\theta ,\sigma $ such
that $x>-\min \left( 1,\theta ,\sigma \right) $, we have%
\begin{equation*}
\Gamma \left( x+1\right) \thicksim \sqrt{2\pi }M\left( x+\theta ,x+\sigma
\right) ^{x+1/2}e^{-M\left( x+\theta ,x+\sigma \right) }e^{r\left( x\right) }%
\text{, as }x\rightarrow \infty .
\end{equation*}
\end{theorem}

\begin{proof}
Since $\lim_{x\rightarrow \infty }r\left( x\right) =0$, the desired result
is equivalent to%
\begin{equation*}
\lim_{x\rightarrow \infty }\left( \ln \Gamma \left( x+1\right) -\ln \sqrt{%
2\pi }-\left( x+\frac{1}{2}\right) \ln M\left( x+\theta ,x+\sigma \right)
+M\left( x+\theta ,x+\sigma \right) \right) =0.
\end{equation*}%
Similarly, it suffices to prove that 
\begin{eqnarray*}
D_{3} &:&=\lim_{x\rightarrow \infty }\left( \left( x+\frac{1}{2}\right) \ln 
\frac{M\left( x+\theta ,x+\sigma \right) }{x+1/2}-\left( M\left( x+\theta
,x+\sigma \right) -\left( x+\frac{1}{2}\right) \right) \right) \\
&=&\lim_{x\rightarrow \infty }\left( \left( M\left( x+\theta ,x+\sigma
\right) -\left( x+\frac{1}{2}\right) \right) \times \left( \frac{1}{L\left(
y,1\right) }-1\right) \right) =0,
\end{eqnarray*}%
where $L\left( a,b\right) $ is the logarithmic mean of positive $a$ and $b$, 
$y=M\left( x+\theta ,x+\sigma \right) /\left( x+1/2\right) $.

Now we first show that%
\begin{equation*}
D_{4}:=M\left( x+\theta ,x+\sigma \right) -\left( x+\frac{1}{2}\right)
\end{equation*}%
is bounded. In fact, by the property of mean we see that%
\begin{equation*}
x+\min \left( \theta ,\sigma \right) -\left( x+\frac{1}{2}\right)
<D_{4}<x+\max \left( \theta ,\sigma \right) -\left( x+\frac{1}{2}\right)
\end{equation*}%
that is,%
\begin{equation*}
\min \left( \theta ,\sigma \right) -\frac{1}{2}<D_{4}<\max \left( \theta
,\sigma \right) -\frac{1}{2}.
\end{equation*}%
It remains to prove that%
\begin{equation*}
\lim_{x\rightarrow \infty }D_{5}:=\lim_{x\rightarrow \infty }\left( \frac{1}{%
L\left( y,1\right) }-1\right) =0.
\end{equation*}%
Since%
\begin{equation*}
\frac{x+\min \left( \theta ,\sigma \right) }{x+1/2}<y=\frac{M\left( x+\theta
,x+\sigma \right) }{x+1/2}<\frac{x+\max \left( \theta ,\sigma \right) }{x+1/2%
},
\end{equation*}%
so we have $\lim_{x\rightarrow \infty }y=1$. This together with%
\begin{equation*}
\min \left( y,1\right) \leq L\left( y,1\right) \leq \max \left( y,1\right)
\end{equation*}%
yields $\lim_{x\rightarrow \infty }L\left( y,1\right) =1$, and therefore, $%
\lim_{x\rightarrow \infty }D_{5}=0$.

This completes the proof.
\end{proof}

\begin{theorem}
\label{MT-p2=p3=x+1/2}Let $K:\left( -\infty ,\infty \right) \times \left(
-\infty ,\infty \right) \rightarrow \left( -\infty ,\infty \right) $ be a
symmetric, homogeneous and twice differentiable mean and let $r$ be defined
on $\left( 0,\infty \right) $ satisfying $\lim_{x\rightarrow \infty }r\left(
x\right) =0$. Then for fixed real numbers $\epsilon ,\epsilon ^{\ast }$ with 
$\epsilon +\epsilon ^{\ast }=1$, we have%
\begin{equation*}
\Gamma \left( x+1\right) \thicksim \sqrt{2\pi }\left( x+\frac{1}{2}\right)
^{K(x+\epsilon ,x+\epsilon ^{\ast })}e^{-\left( x+1/2\right) }e^{r\left(
x\right) }\text{, as }x\rightarrow \infty
\end{equation*}
\end{theorem}

\begin{proof}
Due to $\lim_{x\rightarrow \infty }r\left( x\right) =0$, the result in
question is equivalent to%
\begin{equation*}
\lim_{x\rightarrow \infty }\left( \ln \Gamma \left( x+1\right) -\ln \sqrt{%
2\pi }-K\left( x+\epsilon ,x+\epsilon ^{\ast }\right) \ln \left( x+\frac{1}{2%
}\right) +\left( x+\frac{1}{2}\right) \right) =0.
\end{equation*}%
Clearly, we only need to prove that 
\begin{equation*}
D_{6}:=\lim_{x\rightarrow \infty }\left( K\left( x+\epsilon ,x+\epsilon
^{\ast }\right) -\left( x+\frac{1}{2}\right) \right) \ln \left( x+\frac{1}{2}%
\right) =0.
\end{equation*}%
By the homogeneity of $K$, we get 
\begin{eqnarray*}
&&D_{6}\!\overset{1/x=t}{=\!=\!=}\lim_{t\rightarrow 0^{+}}\frac{K\left(
1+\epsilon t,1+\epsilon ^{\ast }t\right) -\left( 1+t/2\right) }{t}\left( \ln
\left( 1+\frac{t}{2}\right) -\ln t\right) \\
&=&\lim_{t\rightarrow 0^{+}}\frac{K\left( 1+\epsilon t,1+\epsilon ^{\ast
}t\right) -\left( 1+t/2\right) }{t^{2}}\lim_{t\rightarrow 0^{+}}\left( t\ln
\left( 1+\frac{t}{2}\right) -t\ln t\right) =0,
\end{eqnarray*}%
where the first limit, by L'Hospital's rule, is equal to%
\begin{eqnarray*}
&&\lim_{t\rightarrow 0^{+}}\frac{\epsilon K_{x}\left( 1+\epsilon
t,1+\epsilon ^{\ast }t\right) +\epsilon ^{\ast }K_{y}\left( 1+\epsilon
t,1+\epsilon ^{\ast }t\right) -1/2}{2t} \\
&=&\lim_{t\rightarrow 0^{+}}\frac{\epsilon ^{2}K_{xx}\left( 1+\epsilon
t,1+\epsilon ^{\ast }t\right) +2\epsilon \epsilon ^{\ast }K_{xy}\left(
1+\epsilon t,1+\epsilon ^{\ast }t\right) +\epsilon ^{\ast }K_{yy}\left(
1+\epsilon t,1+\epsilon ^{\ast }t\right) }{2} \\
&=&\frac{\epsilon ^{2}K_{xx}\left( 1,1\right) +2\epsilon \epsilon ^{\ast
}K_{xy}\left( 1,1\right) +\epsilon ^{\ast }K_{yy}\left( 1,1\right) }{2}=-%
\frac{\left( 2\epsilon -1\right) ^{2}}{2}K_{xy}\left( 1,1\right) ,
\end{eqnarray*}%
while the second one is clearly equal to zero.

The proof ends.
\end{proof}

By the above three theorems, the following assertion is immediate.

\begin{corollary}
\label{MCg-form1}Suppose that

(i) the function $K:\mathbb{R}^{2}\rightarrow \mathbb{R}$ is a symmetric,
homogeneous and twice differentiable mean;

(ii) the functions $M:\left( 0,\infty \right) \times \left( 0,\infty \right)
\rightarrow \left( 0,\infty \right) $ and $N:\mathbb{R}^{2}\rightarrow 
\mathbb{R}$ are two symmetric, homogeneous, and differentiable means;

(iii) the function $r:\left( 0,\infty \right) \rightarrow \left( -\infty
,\infty \right) $ satisfies $\lim_{x\rightarrow \infty }r\left( x\right) =0$.

Then for fixed real numbers $\epsilon ,\epsilon ^{\ast },\theta ,\theta
^{\ast },\sigma ,\sigma ^{\ast }$ with $\epsilon +\epsilon ^{\ast }=\theta
+\theta ^{\ast }=\sigma +\sigma ^{\ast }=1$ such that $x>-\min \left(
1,\theta ,\theta ^{\ast }\right) $, we have%
\begin{equation*}
\Gamma \left( x+1\right) \thicksim \sqrt{2\pi }M\left( x+\theta ,x+\theta
^{\ast }\right) ^{K\left( x+\epsilon ,x+\epsilon ^{\ast }\right)
}e^{-N\left( x+\sigma ,x+\sigma ^{\ast }\right) }e^{r\left( x\right) },\text{
as }x\rightarrow \infty .
\end{equation*}
\end{corollary}

\begin{corollary}
\label{MCg-form2}Suppose that

(i) the function $K:\left( -\infty ,\infty \right) ^{2}\rightarrow \left(
-\infty ,\infty \right) $ is a symmetric, homogeneous and twice
differentiable mean;

(ii) the functions $M,N:\left( 0,\infty \right) ^{2}\rightarrow \left(
0,\infty \right) $ are two means;

(iii) the function $r:\left( 0,\infty \right) \rightarrow \left( -\infty
,\infty \right) $ satisfies $\lim_{x\rightarrow \infty }r\left( x\right) =0$.

Then for fixed real numbers $\epsilon ,\epsilon ^{\ast },\theta ,\sigma $
with $\epsilon +\epsilon ^{\ast }=1$ such that $x>-\min \left( 1,\theta
,\sigma \right) $, we have 
\begin{equation*}
\Gamma \left( x+1\right) \thicksim \sqrt{2\pi }M\left( x+\theta ,x+\sigma
\right) ^{K\left( x+\epsilon ,x+\epsilon ^{\ast }\right) }e^{-M\left(
x+\theta ,x+\sigma \right) }e^{r\left( x\right) },\text{ as }x\rightarrow
\infty .
\end{equation*}
\end{corollary}

Further, it is obvious that our ideas constructing asymptotic formulas for
the gamma function in terms of bivariate means can be extended to the psi
and polygamma functions.

\begin{theorem}
Let $M:\left( 0,\infty \right) ^{2}\rightarrow \left( 0,\infty \right) $ be
a mean and let $r$ be defined on $\left( 0,\infty \right) $ satisfying $%
\lim_{x\rightarrow \infty }r\left( x\right) =0$. Then for fixed real numbers 
$\theta $, $\sigma $ such that $x>-\min \left( 1,\theta ,\sigma \right) $,
the asymptotic formula for the psi function%
\begin{equation*}
\psi \left( x+1\right) \thicksim \ln M\left( x+\theta ,x+\sigma \right)
+r\left( x\right)
\end{equation*}%
holds as $x\rightarrow \infty $.
\end{theorem}

\begin{proof}
It suffices to prove%
\begin{equation*}
\lim_{x\rightarrow \infty }\left( \psi \left( x+1\right) -\ln M\left(
x+\theta ,x+\sigma \right) \right) =0.
\end{equation*}%
Since $M$ is a mean, we have $x+\min \left( \theta ,\sigma \right) \leq
M\left( x+\theta ,x+\sigma \right) \leq x+\max \left( \theta ,\sigma \right) 
$, and so%
\begin{equation*}
\psi \left( x+1\right) -\ln \left( x+\max \left( \theta ,\sigma \right)
\right) <\psi \left( x+1\right) -\ln M\left( x+\theta ,x+\sigma \right)
<\psi \left( x+1\right) -\ln \left( x+\min \left( \theta ,\sigma \right)
\right) ,
\end{equation*}%
which yields the inquired result due to%
\begin{equation*}
\lim_{x\rightarrow \infty }\left( \psi \left( x+1\right) -\ln \left( x+\max
\left( \theta ,\sigma \right) \right) \right) =\lim_{x\rightarrow \infty
}\left( \psi \left( x+1\right) -\ln \left( x+\min \left( \theta ,\sigma
\right) \right) \right) =0.
\end{equation*}
\end{proof}

\begin{theorem}
Let $M:\left( 0,\infty \right) ^{2}\rightarrow \left( 0,\infty \right) $ be
a mean and let $r$ be defined on $\left( 0,\infty \right) $ satisfying $%
\lim_{x\rightarrow \infty }r\left( x\right) =0$. Then for fixed real numbers 
$\theta ,\sigma $ such that $x>-\min \left( 1,\theta ,\sigma \right) $, the
asymptotic formula for the polygamma function%
\begin{equation*}
\psi ^{(n)}\left( x+1\right) \thicksim \frac{\left( -1\right) ^{n-1}\left(
n-1\right) !}{M^{n}\left( x+\theta ,x+\sigma \right) }+r\left( x\right)
\end{equation*}%
holds as $x\rightarrow \infty $.
\end{theorem}

\begin{proof}
It suffices to show%
\begin{equation*}
\lim_{x\rightarrow \infty }\left( \left( -1\right) ^{n-1}\psi ^{(n)}\left(
x+1\right) -\frac{\left( n-1\right) !}{M^{n}\left( x+\theta ,x+\sigma
\right) }\right) =0.
\end{equation*}%
For this purpose, we utilize a known double inequality that for $k\in 
\mathbb{N}$%
\begin{equation*}
\frac{(k-1)!}{x^{k}}+\frac{k!}{2x^{k+1}}<\left( -1\right) ^{k+1}\psi
^{(k)}\left( x\right) <\frac{(k-1)!}{x^{k}}+\frac{k!}{x^{k+1}}
\end{equation*}%
holds on $(0,\infty )$ proved by Guo and Qi in \cite[Lemma 3]%
{Guo-BKMS-47(1)-2010} to get%
\begin{equation*}
\frac{k!}{2x^{k+1}}<\left( -1\right) ^{k+1}\psi ^{(k)}\left( x\right) -\frac{%
(k-1)!}{x^{k}}<\frac{k!}{x^{k+1}}.
\end{equation*}%
This implies that%
\begin{equation}
\lim_{x\rightarrow \infty }\left( \left( -1\right) ^{k-1}\psi ^{(k)}\left(
x\right) -\frac{(k-1)!}{x^{k}}\right) =0.  \label{GQ}
\end{equation}%
On the other hand, without loss of generality, we assume that $\theta \leq
\sigma $. By the property of mean, we see that 
\begin{equation*}
x+\theta \leq M\left( x+\theta ,x+\sigma \right) \leq x+\sigma ,
\end{equation*}%
and so%
\begin{eqnarray*}
\left( -1\right) ^{n-1}\psi ^{(n)}\left( x+1\right) -\frac{\left( n-1\right)
!}{\left( x+\theta \right) ^{n}} &<&\left( -1\right) ^{n-1}\psi ^{(n)}\left(
x+1\right) -\frac{\left( n-1\right) !}{M^{n}\left( x+\theta ,x+\sigma
\right) } \\
&<&\left( -1\right) ^{n-1}\psi ^{(n)}\left( x+1\right) -\frac{1}{\left(
x+\sigma \right) ^{n}}.
\end{eqnarray*}%
Then, by (\ref{GQ}), for $a=\theta ,\sigma $, we get%
\begin{eqnarray*}
&&\left( -1\right) ^{n-1}\psi ^{(n)}\left( x+1\right) -\frac{\left(
n-1\right) !}{\left( x+a\right) ^{n}} \\
&=&\left( \left( -1\right) ^{n-1}\psi ^{(n)}\left( x+1\right) -\frac{(n-1)!}{%
\left( x+1\right) ^{n}}\right) +\left( \frac{(n-1)!}{\left( x+1\right) ^{n}}-%
\frac{\left( n-1\right) !}{\left( x+a\right) ^{n}}\right) \\
&\rightarrow &0+0=0\text{, as }x\rightarrow \infty ,
\end{eqnarray*}%
which gives the desired result.

Thus we complete the proof.
\end{proof}

\section{Examples}

In this section, we will list some examples to illustrate applications of
Theorems \ref{MT-p2><p3} and \ref{MT-p2=p3}. To this end, we first recall
the arithmetic mean $A$, geometric mean $G$, and identric (exponential) mean 
$I$ of two positive numbers $a$ and $b$ defined by%
\begin{eqnarray*}
A\left( a,b\right) &=&\frac{a+b}{2}\text{, \ \ \ }G\left( a,b\right) =\sqrt{%
ab}, \\
\mathcal{I}\left( a,b\right) &=&\left( b^{b}/a^{a}\right) ^{1/\left(
b-a\right) }/e\text{ if }a\neq b\text{ and }I\left( a,a\right) =a,
\end{eqnarray*}%
(see \cite{Stolarsky-MM-48-1975}, \cite{Yang-MPT-4-1987}). Clearly, these
means are symmetric and homogeneous. Another possible mean is defined by%
\begin{equation}
H_{^{p_{k};q_{k}}}^{n,n-1}\left( a,b\right) =\frac{%
\sum_{k=0}^{n}p_{k}a^{k}b^{n-k}}{\sum_{k=0}^{n-1}q_{k}a^{k}b^{n-1-k}},
\label{H^n,n-1}
\end{equation}%
where%
\begin{equation}
\sum_{k=0}^{n}p_{k}=\sum_{k=0}^{n-1}q_{k}=1.  \label{pk-qk1}
\end{equation}%
It is clear that $H_{^{p_{k};q_{k}}}^{n,n-1}\left( a,b\right) $ is
homogeneous and satisfies $H_{^{p_{k};q_{k}}}^{n,n-1}\left( a,a\right) =a$.

When $p_{k}=p_{n-k}$ and $q_{k}=q_{n-1-k}$, we denote $%
H_{^{p_{k};q_{k}}}^{n,n-1}\left( a,b\right) $ by $S_{^{p_{k};q_{k}}}^{n,n-1}%
\left( a,b\right) $, which can be expressed as%
\begin{equation}
S_{^{p_{k};q_{k}}}^{n,n-1}\left( a,b\right) =\frac{\sum_{k=0}^{[n/2]}p_{k}%
\left( ab\right) ^{k}\left( a^{n-2k}+b^{n-2k}\right) }{\sum_{k=0}^{[\left(
n-1\right) /2]}q_{k}\left( ab\right) ^{k}\left( a^{n-1-2k}+b^{n-1-2k}\right) 
},  \label{S^n,n-1}
\end{equation}%
where $p_{k}$ and $q_{k}$ satisfy%
\begin{equation}
\sum_{k=0}^{[n/2]}\left( 2p_{k}\right) =\sum_{k=0}^{[\left( n-1\right)
/2]}\left( 2q_{k}\right) =1,  \label{pk-qk2}
\end{equation}%
$[x]$ denotes the integer part of real number $x$. Evidently, $%
S_{^{p_{k};q_{k}}}^{n,n-1}$ is symmetric and homogeneous, and $%
S_{^{p_{k};q_{k}}}^{n,n-1}\left( a,a\right) =a$. But $%
H_{^{p_{k};q_{k}}}^{n,n-1}\left( a,b\right) $ and $%
S_{^{p_{k};q_{k}}}^{n,n-1}\left( a,b\right) $ are not always means of $a$
and $b$. For instance, when $p=2/3$, 
\begin{equation*}
S_{^{p;1/2}}^{2,1}\left( a,b\right) =\frac{pa^{2}+pb^{2}+\left( 1-2p\right)
ab}{\left( a+b\right) /2}=\frac{2}{3}\frac{2a^{2}+2b^{2}-ab}{a+b}>\max (a,b)
\end{equation*}%
in the case of $\max (a,b)>4\min \left( a,b\right) $. Indeed, it is easy to
prove that $S_{^{p;1/2}}^{2,1}\left( a,b\right) $ is a mean if and only if $%
p\in \lbrack 0,1/2]$.

Secondly, we recall the so-called completely monotone functions. A function $%
f$ is said to be completely monotonic on an interval $I$ , if $f$ has
derivatives of all orders on $I$ and satisfies

\begin{equation}
(-1)^{n}f^{(n)}(x)\geq 0\text{ for all }x\in I\text{ and }n=0,1,2,....
\label{cm}
\end{equation}

If the inequality (\ref{cm}) is strict, then $f$ is said to be strictly
completely monotonic on $I$. It is known (Bernstein's Theorem) that $f$ is
completely monotonic on $(0,\infty )$ if and only if

\begin{equation*}
f(x)=\int_{0}^{\infty }e^{-xt}d\mu \left( t\right) ,
\end{equation*}%
where $%
\mu
$ is a nonnegative measure on $[0,\infty )$ such that the integral converges
for all $x>0$, see \cite[p. 161]{Widder-PUPP-1941}.

\begin{example}
Let 
\begin{eqnarray*}
K\left( a,b\right) &=&N\left( a,b\right) =A\left( a,b\right) =\frac{a+b}{2},
\\
M\left( a,b\right) &=&A^{2/3}\left( a,b\right) G^{1/3}\left( a,b\right)
=\left( \frac{a+b}{2}\right) ^{2/3}\left( \sqrt{ab}\right) ^{1/3}
\end{eqnarray*}%
and $\theta =\sigma =0$ in Theorem \ref{MT-p2><p3}. Then we can obtain an
asymptotic formulas for the gamma function as follows.%
\begin{eqnarray*}
\ln \Gamma (x+1) &\thicksim &\frac{1}{2}\ln 2\pi +\left( x+\frac{1}{2}%
\right) \ln \left( \left( x+\frac{1}{2}\right) ^{2/3}\left( \sqrt{x\left(
x+1\right) }\right) ^{1/3}\right) -\left( x+\frac{1}{2}\right) \\
&=&\frac{1}{2}\ln 2\pi +\frac{2}{3}\left( x+\frac{1}{2}\right) \ln \left( x+%
\frac{1}{2}\right) +\frac{1}{6}\left( x+\frac{1}{2}\right) \ln x \\
&&+\frac{1}{6}\left( x+\frac{1}{2}\right) \ln \left( x+1\right) -\left( x+%
\frac{1}{2}\right) ,\text{ as }x\rightarrow \infty .
\end{eqnarray*}
\end{example}

Further, we can prove

\begin{proposition}
For $x>0$, the function 
\begin{eqnarray*}
f_{1}(x) &=&\ln \Gamma (x+1)-\frac{1}{2}\ln 2\pi -\frac{2}{3}\left( x+\frac{1%
}{2}\right) \ln \left( x+\frac{1}{2}\right) -\frac{1}{6}\left( x+\frac{1}{2}%
\right) \ln x \\
&&-\frac{1}{6}\left( x+\frac{1}{2}\right) \ln \left( x+1\right) +\left( x+%
\frac{1}{2}\right)
\end{eqnarray*}%
is a completely monotone function.
\end{proposition}

\begin{proof}
Differentiating and utilizing the relations 
\begin{equation}
\psi (x)=\int_{0}^{\infty }\left( \frac{e^{-t}}{t}-\frac{e^{-xt}}{1-e^{-t}}%
\right) dt\text{ \ and \ }\ln x=\int_{0}^{\infty }\frac{e^{-t}-e^{-xt}}{t}dt
\label{psi-ln}
\end{equation}%
yield%
\begin{eqnarray*}
f_{1}^{\prime }(x) &=&\psi \left( x+1\right) -\frac{1}{6}\ln \left(
x+1\right) -\frac{1}{6}\ln x-\frac{2}{3}\ln \left( x+\frac{1}{2}\right) +%
\frac{1}{12\left( x+1\right) }-\frac{1}{12x} \\
&=&\int_{0}^{\infty }\left( \frac{e^{-t}}{t}-\frac{e^{-\left( x+1\right) t}}{%
1-e^{-t}}\right) dt-\int_{0}^{\infty }\frac{e^{-t}-e^{-xt}}{6t}%
dt-\int_{0}^{\infty }\frac{e^{-t}-e^{-\left( x+1\right) t}}{6t}dt \\
&&-\int_{0}^{\infty }\frac{2\left( e^{-t}-e^{-\left( x+1/2\right) t}\right) 
}{3t}dt+\frac{1}{12}\int_{0}^{\infty }e^{-\left( x+1\right) t}dt-\frac{1}{12}%
\int_{0}^{\infty }e^{-xt}dt \\
&=&\int_{0}^{\infty }e^{-xt}\left( \frac{1}{6t}+\frac{e^{-t}}{6t}+\frac{%
2e^{-t/2}}{3t}-\frac{e^{-t/2}}{1-e^{-t}}+\frac{1}{12}\left( e^{-t}-1\right)
\right) dt \\
&=&\int_{0}^{\infty }e^{-xt}e^{-t/2}\left( \frac{\cosh \left( t/2\right) }{3t%
}+\frac{2}{3t}-\frac{1}{2\sinh \left( t/2\right) }-\frac{1}{6}\sinh \frac{t}{%
2}\right) dt \\
&:&=\int_{0}^{\infty }e^{-xt}e^{-t/2}u\left( \frac{t}{2}\right) dt,
\end{eqnarray*}%
where%
\begin{equation*}
u\left( t\right) =\frac{\cosh t}{6t}+\frac{1}{3t}-\frac{1}{2\sinh t}-\frac{1%
}{6}\sinh t.
\end{equation*}%
Factoring and expanding in power series lead to%
\begin{eqnarray*}
u\left( t\right) &=&-\frac{t\cosh 2t-\sinh 2t-4\sinh t+5t}{12t\sinh t} \\
&=&-\frac{\sum_{n=1}^{\infty }\frac{2^{2n-2}t^{2n-1}}{\left( 2n-2\right) !}%
-\sum_{n=1}^{\infty }\frac{2^{2n-1}t^{2n-1}}{\left( 2n-1\right) !}%
-4\sum_{n=1}^{\infty }\frac{t^{2n-1}}{\left( 2n-1\right) !}+5t}{12t\sinh
\left( t/2\right) } \\
&=&-\frac{\sum_{n=3}^{\infty }\frac{\left( 2n-3\right) 2^{2n-2}-4}{\left(
2n-1\right) !}t^{2n-1}}{12t\sinh t}<0
\end{eqnarray*}%
for $t>0$. This reveals that $-f_{1}^{\prime }$ is a completely monotone
function, which together with $f_{1}(x)>\lim_{x\rightarrow \infty
}f_{1}(x)=0 $ leads us to the desired result.
\end{proof}

Using the decreasing property of $f_{1}$ on $\left( 0,\infty \right) $ and
notice that%
\begin{equation*}
f_{1}(1)=\ln \frac{2^{3/4}e^{3/2}}{3\sqrt{2\pi }}\text{ \ and \ }%
f_{1}(\infty )=0
\end{equation*}%
we immediately get

\begin{corollary}
For $n\in \mathbb{N}$, it is true that%
\begin{equation*}
\sqrt{2\pi }\left( \frac{(n+1/2)^{4}n\left( n+1\right) }{e^{6}}\right)
^{\left( n+1/2\right) /6}<n!<\frac{2^{3/4}e^{3/2}}{3}\left( \frac{%
(n+1/2)^{4}n\left( n+1\right) }{e^{6}}\right) ^{\left( n+1/2\right) /6},
\end{equation*}%
with the optimal constants $\sqrt{2\pi }\approx 2.5066$ and $%
2^{3/4}e^{3/2}/3\approx 2.5124$.
\end{corollary}

\begin{example}
Let 
\begin{eqnarray*}
K\left( a,b\right) &=&N\left( a,b\right) =A\left( a,b\right) =\frac{a+b}{2},
\\
M\left( a,b\right) &=&\mathcal{I}\left( a,b\right) =\left(
b^{b}/a^{a}\right) ^{1/\left( b-a\right) }/e\text{ if }a\neq b\text{ and }%
I\left( a,a\right) =a
\end{eqnarray*}%
and $\theta =0$ in Theorem \ref{MT-p2><p3}. Then we get the asymptotic
formulas:%
\begin{equation*}
\ln \Gamma (x+1)\thicksim \frac{1}{2}\ln 2\pi +\left( x+\frac{1}{2}\right)
\left( (x+1)\ln (x+1)-x\ln x-1\right) -\left( x+\frac{1}{2}\right) ,
\end{equation*}%
as $x\rightarrow \infty $.
\end{example}

And, we have

\begin{proposition}
For $x>0$, the function 
\begin{equation*}
f_{2}(x)=\ln \Gamma (x+1)-\frac{1}{2}\ln 2\pi -\left( x+\frac{1}{2}\right)
\left( (x+1)\ln (x+1)-x\ln x-1\right) +x+\frac{1}{2}
\end{equation*}%
is a completely monotone function.
\end{proposition}

\begin{proof}
Differentiation gives%
\begin{eqnarray*}
f_{2}^{\prime }(x) &=&\psi \left( x+1\right) -\left( 2x+\frac{3}{2}\right)
\ln \left( x+1\right) +\left( 2x+\frac{1}{2}\right) \ln x+2, \\
f_{2}^{\prime \prime }(x) &=&\psi ^{\prime }\left( x+1\right) -2\ln \left(
x+1\right) +2\ln x+\frac{1}{2\left( x+1\right) }+\frac{1}{2x}.
\end{eqnarray*}%
Application of the relations (\ref{psi-ln}), $f_{2}^{\prime \prime }(x)$ can
be expressed as%
\begin{eqnarray*}
f_{2}^{\prime \prime }(x) &=&\int_{0}^{\infty }t\frac{e^{-\left( x+1\right)
t}}{1-e^{-t}}dt-2\int_{0}^{\infty }\frac{e^{-xt}-e^{-\left( x+1\right) t}}{t}%
dt+\frac{1}{2}\int_{0}^{\infty }\left( e^{-\left( x+1\right)
t}+e^{-xt}\right) dt \\
&=&\int_{0}^{\infty }e^{-xt}\left( \frac{te^{-t}}{1-e^{-t}}-2\frac{1-e^{-t}}{%
t}+\frac{1}{2}\left( e^{-t}+1\right) \right) dt \\
&=&\int_{0}^{\infty }e^{-xt}e^{-t/2}\left( \frac{t}{2\sinh \left( t/2\right) 
}-4\frac{\sinh \left( t/2\right) }{t}+\cosh \frac{t}{2}\right) dt \\
&:&=\int_{0}^{\infty }e^{-xt}e^{-t/2}v\left( \tfrac{t}{2}\right) dt,
\end{eqnarray*}%
where%
\begin{equation*}
v\left( t\right) =\frac{t}{\sinh t}-2\frac{\sinh t}{t}+\cosh t.
\end{equation*}%
Employing hyperbolic version of Wilker inequality proved in \cite%
{Zhu-MIA-10(4)-2007} (also see \cite{Zhu-AAA-485842-2009}, \cite%
{Yang-JIA-2014-166})%
\begin{equation*}
\left( \frac{t}{\sinh t}\right) ^{2}+\frac{t}{\tanh t}>2,
\end{equation*}%
we get%
\begin{equation*}
\frac{\sinh t}{t}v\left( t\right) =\left( \frac{t}{\sinh t}\right) ^{2}+%
\frac{t}{\tanh t}-2>0,
\end{equation*}%
and so $f_{2}^{\prime \prime }(x)$ is complete monotone for $x>0$. Hence, $%
f_{2}^{\prime }(x)<\lim_{x\rightarrow \infty }f_{2}^{\prime }(x)=0$, and
then, $f_{2}(x)>\lim_{x\rightarrow \infty }f_{2}(x)=0$, which indicate that $%
f_{2}$ is complete monotone for $x>0$.

This completes the proof.
\end{proof}

The decreasing property of $f_{2}$ on $\left( 0,\infty \right) $ and the
facts that%
\begin{equation*}
f_{2}\left( 0^{+}\right) =\ln \frac{e}{\sqrt{2\pi }}\text{, \ }f_{2}\left(
1\right) =\ln \frac{e^{3}}{8}\text{, \ }f_{2}\left( \infty \right) =0
\end{equation*}%
give the following

\begin{corollary}
For $x>0$, the sharp double inequality%
\begin{equation*}
\sqrt{2\pi }e^{-2x-1}\frac{(x+1)^{(x+1)\left( x+1/2\right) }}{x^{x\left(
x+1/2\right) }}<\Gamma (x+1)<e^{-2x}\frac{(x+1)^{(x+1)\left( x+1/2\right) }}{%
x^{x\left( x+1/2\right) }}
\end{equation*}%
holds.

For $n\in \mathbb{N}$, it holds that%
\begin{equation*}
\sqrt{2\pi }e^{-2n-1}\frac{(n+1)^{(n+1)\left( n+1/2\right) }}{n^{n\left(
n+1/2\right) }}<n!<\frac{e^{3}}{8}e^{-2n-1}\frac{(n+1)^{(n+1)\left(
n+1/2\right) }}{n^{n\left( n+1/2\right) }}
\end{equation*}%
with the best constants $\sqrt{2\pi }\approx 2.5066$ and $e^{3}/8\approx
2.5107$.
\end{corollary}

\begin{example}
\label{E-M3,2}Let%
\begin{eqnarray*}
K\left( a,b\right) &=&N\left( a,b\right) =A\left( a,b\right) =\frac{a+b}{2},
\\
M\left( a,b\right) &=&M_{^{p;q}}^{3,2}\left( a,b\right) =\frac{%
pa^{3}+pb^{3}+\left( 1/2-p\right) a^{2}b+\left( 1/2-p\right) ab^{2}}{%
qa^{2}+qb^{2}+(1-2q)ab} \\
&=&\frac{a+b}{2}\frac{2pa^{2}+2pb^{2}+\left( 1-4p\right) ab}{%
qa^{2}+qb^{2}+\left( 1-2q\right) ab}
\end{eqnarray*}%
and $\theta =0$ in Theorem \ref{MT-p2><p3}, where $p$ and $q$ are parameters
to be determined. Then, we have%
\begin{eqnarray*}
K\left( x,x+1\right) &=&N\left( x,x+1\right) =x+\frac{1}{2}, \\
M\left( x,x+1\right) &=&S_{^{p;q}}^{3,2}\left( x,x+1\right) =\left(
x+1/2\right) \frac{x^{2}+x+2p}{x^{2}+x+q}.
\end{eqnarray*}%
Straightforward computations give%
\begin{eqnarray*}
\lim_{x\rightarrow \infty }\tfrac{\ln \Gamma (x+1)-\ln \sqrt{2\pi }-\left(
x+1/2\right) \ln M_{p;q}^{3,2}\left( x,x+1\right) +x+1/2}{x^{-1}} &=&q-2p-%
\frac{1}{24}, \\
\lim_{x\rightarrow \infty }\tfrac{\ln \Gamma (x+1)-\ln \sqrt{2\pi }-\left(
x+1/2\right) \ln M_{p;2p+1/24}^{3,2}\left( x,x+1\right) +x+1/2}{x^{-3}} &=&-%
\frac{160}{1920}\left( p-\frac{23}{160}\right) ,
\end{eqnarray*}%
and solving the equation set%
\begin{equation*}
q-2p-\frac{1}{24}=0\text{ and }-\frac{160}{1920}\left( p-\frac{23}{160}%
\right) =0
\end{equation*}%
leads to%
\begin{equation*}
p=\frac{23}{160},q=\frac{79}{240}.
\end{equation*}%
And then,%
\begin{equation*}
M\left( x,x+1\right) =\left( x+\frac{1}{2}\right) \frac{x^{2}+x+\frac{23}{80}%
}{x^{2}+x+\frac{79}{240}}.
\end{equation*}%
It is easy to check that $S_{^{p;q}}^{3,2}\left( a,b\right) $ is a symmetric
and homogeneous mean of positive numbers $a$ and $b$ for $p=23/160$, $%
q=79/240$. Hence, by Theorem \ref{MT-p2><p3}, we have the optimal asymptotic
formula for the gamma function%
\begin{equation*}
\ln \Gamma (x+1)\thicksim \frac{1}{2}\ln 2\pi +\left( x+\frac{1}{2}\right)
\ln \tfrac{\left( x+1/2\right) \left( x^{2}+x+23/80\right) }{x^{2}+x+79/240}%
-\left( x+\frac{1}{2}\right) ,
\end{equation*}%
as $x\rightarrow \infty $, and%
\begin{equation*}
\lim_{x\rightarrow \infty }\tfrac{\ln \Gamma (x+1)-\ln \sqrt{2\pi }-\left(
x+1/2\right) \ln \tfrac{\left( x+1/2\right) \left( x^{2}+x+23/80\right) }{%
x^{2}+x+79/240}+x+1/2}{x^{-5}}=-\tfrac{18\,029}{29\,030\,400}.
\end{equation*}
\end{example}

Also, this asymptotic formula have a well property.

\begin{proposition}
For $x>-1/2$, the function $f_{3}$ defined by%
\begin{equation}
f_{3}\left( x\right) =\ln \Gamma (x+1)-\frac{1}{2}\ln 2\pi -\left( x+\frac{1%
}{2}\right) \ln \tfrac{\left( x+1/2\right) \left( x^{2}+x+23/80\right) }{%
x^{2}+x+79/240}+\left( x+\frac{1}{2}\right) .  \label{f3}
\end{equation}%
is increasing and concave.
\end{proposition}

\begin{proof}
Differentiation gives%
\begin{eqnarray*}
f_{3}^{\prime }\left( x\right) &=&\psi \left( x+1\right) +\ln \left( x^{2}+x+%
\frac{79}{240}\right) -\ln \left( x^{2}+x+\frac{23}{80}\right) \\
&&-\ln \left( x+\frac{1}{2}\right) -2\frac{\left( x+1/2\right) ^{2}}{%
x^{2}+x+23/80}+2\frac{\left( x+1/2\right) ^{2}}{x^{2}+x+79/240},
\end{eqnarray*}%
\begin{eqnarray*}
f_{3}^{\prime \prime }\left( x\right) &=&\psi ^{\prime }\left( x+1\right) +6%
\frac{x+1/2}{x^{2}+x+79/240}-6\frac{x+1/2}{x^{2}+x+23/80} \\
&&-\frac{1}{x+1/2}+4\frac{\left( x+1/2\right) ^{3}}{\left(
x^{2}+x+23/80\right) ^{2}}-4\frac{\left( x+1/2\right) ^{3}}{\left(
x^{2}+x+79/240\right) ^{2}}.
\end{eqnarray*}%
Denote by $x+1/2=t$ and make use of recursive relation%
\begin{equation}
\psi ^{\left( n\right) }(x+1)-\psi ^{\left( n\right) }(x)=\left( -1\right)
^{n}\frac{n!}{x^{n+1}}  \label{psi-rel.}
\end{equation}%
yield%
\begin{eqnarray*}
&&f_{3}^{\prime \prime }(t+\frac{1}{2})-f_{3}^{\prime \prime }(t-\frac{1}{2})
\\
&=&-\tfrac{1}{\left( t+1/2\right) ^{2}}+6\tfrac{t+1}{\left( t+1\right)
^{2}+19/240}-6\tfrac{t+1}{\left( t+1\right) ^{2}+3/80}-\frac{1}{\left(
t+1\right) }+4\tfrac{\left( t+1\right) ^{3}}{\left( \left( t+1\right)
^{2}+3/80\right) ^{2}} \\
&&-4\tfrac{\left( t+1\right) ^{3}}{\left( \left( t+1\right)
^{2}+19/240\right) ^{2}}-\left( 6\tfrac{t}{t^{2}+19/240}-6\tfrac{t}{%
t^{2}+3/80}-\frac{1}{t}+4\tfrac{t^{3}}{\left( t^{2}+3/80\right) ^{2}}-4%
\tfrac{t^{3}}{\left( t^{2}+19/240\right) ^{2}}\right) \\
&=&\frac{f_{31}\left( t\right) }{t\left( t+1\right) \left( t+\frac{1}{2}%
\right) ^{2}\left( t^{2}+2t+83/80\right) ^{2}\left( t^{2}+3/80\right)
^{2}\left( t^{2}+2t+259/240\right) ^{2}\left( t^{2}+19/240\right) ^{2}},
\end{eqnarray*}%
where%
\begin{eqnarray*}
f_{31}\left( t\right) &=&\tfrac{18\,029}{138\,240}t^{12}+\tfrac{18\,029}{%
23\,040}t^{11}+\tfrac{83\,674\,657}{41\,472\,000}t^{10}+\tfrac{24\,178\,957}{%
8294\,400}t^{9}+\tfrac{34\,366\,211\,867}{13\,271\,040\,000}t^{8}+\tfrac{%
4894\,651\,067}{3317\,760\,000}t^{7} \\
&&+\tfrac{74\,296\,657\,243}{132\,710\,400\,000}t^{6}+\tfrac{%
20\,147\,292\,749}{132\,710\,400\,000}t^{5}+\tfrac{297\,092\,035\,417}{%
9437\,184\,000\,000}t^{4}+\tfrac{66\,777\,391\,051}{14\,155\,776\,000\,000}%
t^{3} \\
&&+\tfrac{295\,012\,866\,563}{566\,231\,040\,000\,000}t^{2}+\tfrac{%
3972\,595\,981}{188\,743\,680\,000\,000}t+\tfrac{166\,825\,684\,249}{%
60\,397\,977\,600\,000\,000} \\
&>&0\text{ for }t=x+1/2>0\text{.}
\end{eqnarray*}

This shows that $f_{3}^{\prime \prime }(t+\frac{1}{2})-f_{3}^{\prime \prime
}(t-\frac{1}{2})>0$, that is, $f_{3}^{\prime \prime }(x+1)-f_{3}^{\prime
\prime }(x)>0$, and so%
\begin{equation*}
f_{3}^{\prime \prime }(x)<f_{3}^{\prime \prime }(x+1)<f_{3}^{\prime \prime
}(x+2)<...<f_{3}^{\prime \prime }(\infty )=0.
\end{equation*}%
It reveals that shows $f_{3}$ is concave on $\left( -1/2,\infty \right) $,
and we conclude that, $f_{3}^{\prime }(x)>\lim_{x\rightarrow \infty
}f_{3}^{\prime }(x)=0$, which proves the desired result.
\end{proof}

As a consequence of the above proposition, we have

\begin{corollary}
For $x>0$, the double inequality%
\begin{equation*}
\sqrt{\tfrac{158e}{69}}\left( \tfrac{x+1/2}{e}\tfrac{x^{2}+x+23/80}{%
x^{2}+x+79/240}\right) ^{x+1/2}<\Gamma (x+1)<\sqrt{2\pi }\left( \tfrac{x+1/2%
}{e}\tfrac{x^{2}+x+23/80}{x^{2}+x+79/240}\right) ^{x+1/2}
\end{equation*}%
holds true, where $\sqrt{158e/69}\approx 2.4949$ and and $\sqrt{2\pi }%
\approx 2.5066$ are the best.

For $n\in \mathbb{N}$, it is true that%
\begin{equation*}
\left( \tfrac{1118e}{1647}\right) ^{3/2}\left( \tfrac{n+1/2}{e}\tfrac{%
n^{2}+n+23/80}{n^{2}+n+79/240}\right) ^{n+1/2}<n!<\sqrt{2\pi }\left( \tfrac{%
n+1/2}{e}\tfrac{n^{2}+n+23/80}{n^{2}+n+79/240}\right) ^{n+1/2}
\end{equation*}%
holds true with the best constants $\left( 1118e/1647\right) ^{3/2}\approx
2.5065$ and $\sqrt{2\pi }\approx 2.5066$.
\end{corollary}

\begin{example}
\label{E-N3,2}Let%
\begin{eqnarray*}
K\left( a,b\right) &=&M\left( a,b\right) =A\left( a,b\right) =\frac{a+b}{2},
\\
N\left( a,b\right) &=&S_{^{p;q}}^{3,2}\left( a,b\right) =\frac{%
pa^{3}+pb^{3}+\left( 1/2-p\right) ab^{2}+\left( 1/2-p\right) a^{2}b}{%
qa^{2}+qb^{2}+\left( 1-2q\right) ab} \\
&=&\frac{a+b}{2}\frac{2pa^{2}+2pb^{2}+\left( 1-4p\right) ab}{%
qa^{2}+qb^{2}+\left( 1-2q\right) ab}
\end{eqnarray*}%
and $\sigma =0$ in Theorem \ref{MT-p2><p3}, where $p$ and $q$ are parameters
to be determined. Direct computations give%
\begin{eqnarray*}
\lim_{x\rightarrow \infty }\tfrac{\ln \Gamma (x+1)-\frac{1}{2}\ln 2\pi
-\left( x+1/2\right) \ln \left( x+1/2\right) +\left( x+1/2\right) \frac{%
x^{2}+x+2p}{x^{2}+x+q}}{x^{-1}} &=&2p-q-\frac{1}{24}, \\
\lim_{x\rightarrow \infty }\tfrac{\ln \Gamma (x+1)-\frac{1}{2}\ln 2\pi
-\left( x+1/2\right) \ln \left( x+1/2\right) +\left( x+1/2\right) \frac{%
x^{2}+x+2p}{x^{2}+x+2p-1/24}}{x^{-3}} &=&\frac{7}{480}-\frac{1}{12}p.
\end{eqnarray*}%
Solving the simultaneous equations%
\begin{eqnarray*}
2p-q-\frac{1}{24} &=&0, \\
\frac{7}{480}-\frac{1}{12}p &=&0
\end{eqnarray*}%
leads to $p=7/40$, $q=37/120$. And then,%
\begin{equation*}
N\left( x,x+1\right) =\left( x+1/2\right) \frac{x^{2}+x+7/20}{x^{2}+x+37/120}%
.
\end{equation*}

An easy verification shows that $S_{^{p;q}}^{3,2}\left( a,b\right) $ is a
symmetric and homogeneous mean of positive numbers $a$ and $b$ for $p=7/40$, 
$q=37/120$. Hence, by Theorem \ref{MT-p2><p3} we get the best asymptotic
formula for the gamma function%
\begin{equation*}
\ln \Gamma (x+1)\thicksim \frac{1}{2}\ln 2\pi +\left( x+\frac{1}{2}\right)
\ln \left( x+\frac{1}{2}\right) -\left( x+\frac{1}{2}\right) \frac{%
x^{2}+x+7/20}{x^{2}+x+37/120},
\end{equation*}%
as $x\rightarrow \infty $. And we have%
\begin{equation*}
\lim_{x\rightarrow \infty }\tfrac{\ln \Gamma (x+1)-\frac{1}{2}\ln 2\pi
-\left( x+1/2\right) \ln \left( x+1/2\right) +\left( x+1/2\right) \frac{%
x^{2}+x+7/20}{x^{2}+x+37/120}}{x^{-5}}=-\frac{1517}{2419\,200}.
\end{equation*}
\end{example}

Now we prove the following assertion related to this asymptotic formula.

\begin{proposition}
Let the function $f_{4}$ be defined on $\left( -1/2,\infty \right) $ by%
\begin{equation*}
f_{4}(x)=\ln \Gamma (x+1)-\tfrac{1}{2}\ln 2\pi -\left( x+\tfrac{1}{2}\right)
\ln (x+\tfrac{1}{2})+\left( x+\tfrac{1}{2}\right) \frac{x^{2}+x+7/20}{%
x^{2}+x+37/120}.
\end{equation*}%
Then $f_{4}$ is increasing and convex on $\left( -1/2,\infty \right) $.
\end{proposition}

\begin{proof}
Differentiation gives%
\begin{eqnarray*}
f_{4}^{\prime }(x) &=&\psi \left( x+1\right) -\ln \left( x+\frac{1}{2}%
\right) +\frac{1}{24}\frac{1}{x^{2}+x+37/120}-\frac{1}{12}\frac{\left(
x+1/2\right) ^{2}}{\left( x^{2}+x+37/120\right) ^{2}}, \\
f_{4}^{\prime \prime }(x) &=&\psi ^{\prime }\left( x+1\right) -\frac{1}{x+1/2%
}-\frac{1}{4}\frac{x+1/2}{\left( x^{2}+x+37/120\right) ^{2}}+\frac{1}{3}%
\frac{\left( x+\frac{1}{2}\right) ^{3}}{\left( x^{2}+x+37/120\right) ^{3}}.
\end{eqnarray*}%
Denote by $x+1/2=t$ and make use of recursive relation (\ref{psi-rel.}) yield%
\begin{eqnarray*}
&&f_{4}^{\prime \prime }(t+\frac{1}{2})-f_{4}^{\prime \prime }(t-\frac{1}{2})
\\
&=&-\tfrac{1}{\left( t+1/2\right) ^{2}}-\frac{1}{t+1}-\frac{1}{4}\frac{t+1}{%
\left( \left( t+1\right) ^{2}+7/120\right) ^{2}}+\frac{1}{3}\frac{\left(
t+1\right) ^{3}}{\left( \left( t+1\right) ^{2}+7/120\right) ^{3}} \\
&&-\left( -\frac{1}{t}-\frac{1}{4}\frac{t}{\left( t^{2}+7/120\right) ^{2}}+%
\frac{1}{3}\frac{t^{3}}{\left( t^{2}+7/120\right) ^{3}}\right) \\
&=&\frac{f_{41}\left( t\right) }{t\left( t+1\right) \left( t+1/2\right)
^{2}\left( t^{2}+7/120\right) ^{3}\left( t^{2}+2t+127/120\right) ^{3}},
\end{eqnarray*}%
where%
\begin{eqnarray*}
f_{41}\left( t\right) &=&\frac{1517}{11\,520}t^{8}+\frac{1517}{2880}t^{7}+%
\frac{161\,087}{192\,000}t^{6}+\frac{387\,883}{576\,000}t^{5}+\frac{%
39\,563\,149}{138\,240\,000}t^{4} \\
&&+\frac{4462\,549}{69\,120\,000}t^{3}+\frac{67\,788\,161}{8294\,400\,000}%
t^{2}+\frac{2794\,421}{8294\,400\,000}t+\frac{702\,595\,369}{%
11\,943\,936\,000\,000} \\
&>&0\text{ for }t=x+1/2>0.
\end{eqnarray*}%
This implies that $f_{4}^{\prime \prime }(t+\frac{1}{2})-f_{4}^{\prime
\prime }(t-\frac{1}{2})>0$, that is, $f_{4}^{\prime \prime
}(x+1)-f_{4}^{\prime \prime }(x)>0$, and so%
\begin{equation*}
f_{4}^{\prime \prime }(x)<f_{4}^{\prime \prime }(x+1)<f_{4}^{\prime \prime
}(x+2)<...<f_{4}^{\prime \prime }(\infty )=0.
\end{equation*}%
It reveals that shows $f_{4}$ is concave on $\left( -1/2,\infty \right) $,
and therefore, $f_{4}^{\prime }(x)>\lim_{x\rightarrow \infty }f_{4}^{\prime
}(x)=0$, which proves the desired result.
\end{proof}

By the increasing property of $f_{4}$ on $\left( -1/2,\infty \right) $ and
the facts%
\begin{equation*}
f_{4}\left( 0\right) =\ln \frac{e^{21/37}}{\sqrt{\pi }}\text{, \ }%
f_{4}\left( 1\right) =\ln \frac{2e^{423/277}}{3\sqrt{3\pi }}\text{, \ }%
f_{4}\left( \infty \right) =0,
\end{equation*}%
we have

\begin{corollary}
For $x>0$, the double inequality%
\begin{equation*}
e^{21/37}\sqrt{2}\left( \tfrac{x+1/2}{\exp \left( \frac{x^{2}+x+7/20}{%
x^{2}+x+37/120}\right) }\right) ^{x+1/2}<\Gamma (x+1)<\sqrt{2\pi }\left( 
\tfrac{x+1/2}{\exp \left( \frac{x^{2}+x+7/20}{x^{2}+x+37/120}\right) }%
\right) ^{x+1/2}
\end{equation*}%
holds, where $e^{21/37}\sqrt{2}\approx 2.4946$ and $\sqrt{2\pi }\approx
2.5066$ are the best.

For $n\in \mathbb{N}$, the double inequality%
\begin{equation*}
e^{423/277}\tfrac{2\sqrt{2}}{3\sqrt{3}}(\tfrac{n+1/2}{e})^{n+1/2}\exp \left(
-\tfrac{1}{24}\tfrac{n+1/2}{n^{2}+n+37/120}\right) <n!<\sqrt{2\pi }(\tfrac{%
n+1/2}{e})^{n+1/2}\exp \left( -\tfrac{1}{24}\tfrac{n+1/2}{n^{2}+n+37/120}%
\right)
\end{equation*}%
holds true with the best constants $2\sqrt{2}e^{423/277}/\left( 3\sqrt{3}%
\right) \approx 2.5065$ and $\sqrt{2\pi }\approx 2.5066$.
\end{corollary}

\begin{example}
\label{E-N4,3}Let%
\begin{eqnarray*}
K\left( a,b\right) &=&M\left( a,b\right) =A\left( a,b\right) =x+1/2, \\
N\left( a,b\right) &=&S_{^{p,q;r}}^{4,3}\left( a,b\right) =\frac{%
pa^{4}+pb^{4}+qa^{3}b+qab^{3}+\left( 1-2p-2q\right) a^{2}b^{2}}{%
ra^{3}+rb^{3}+\left( 1/2-r\right) a^{2}b+\left( 1/2-r\right) ab^{2}}
\end{eqnarray*}%
and $\sigma =0$ in Theorem \ref{MT-p2><p3}. In a similar way, we can
determine that the best parameters satisfy 
\begin{equation*}
r=2p+\frac{1}{2}q-\frac{7}{48}\text{, \ }p=\frac{21}{40}-\frac{7}{4}q\text{,
\ }q=\frac{7303}{35\,280},
\end{equation*}%
which imply%
\begin{equation*}
p=\frac{3281}{20\,160},q=\frac{7303}{35\,280};r=\frac{111}{392}.
\end{equation*}%
Then,%
\begin{equation}
N\left( x,x+1\right) =x+\tfrac{1}{2}+\tfrac{1517}{44\,640}\tfrac{1}{x+1/2}+%
\tfrac{343}{44\,640}\tfrac{x+1/2}{x^{2}+x+111/196}:=N_{4/3}\left(
x,x+1\right) ,  \label{N4/3}
\end{equation}%
In this case, we easily check that $S_{^{p,q;r}}^{4,3}\left( a,b\right) $ is
a mean of $a$ and $b$. Consequently, from Theorem \ref{MT-p2><p3} the
following best asymptotic formula for the gamma function%
\begin{equation*}
\ln \Gamma (x+1)\sim \frac{1}{2}\ln 2\pi +\left( x+1/2\right) \ln
(x+1/2)-N_{4/3}\left( x,x+1\right)
\end{equation*}%
holds true as $x\rightarrow \infty $. And, we have%
\begin{equation*}
\lim_{x\rightarrow \infty }\tfrac{\ln \Gamma (x+1)-\frac{1}{2}\ln 2\pi
-\left( x+1/2\right) \ln \left( x+1/2\right) +N_{4/3}\left( x,x+1\right) }{%
x^{-7}}=\tfrac{10\,981}{31\,610\,880}.
\end{equation*}
\end{example}

We now present the monotonicity and convexity involving this asymptotic
formula.

\begin{proposition}
Let $f_{5}$ defined on $\left( -1/2,\infty \right) $ by%
\begin{equation*}
f_{5}(x)=\ln \Gamma (x+1)-\frac{1}{2}\ln 2\pi -\left( x+1/2\right) \ln
(x+1/2)+N_{4/3}\left( x,x+1\right) ,
\end{equation*}%
where $N_{4/3}\left( x,x+1\right) $ is defined\ by (\ref{N4/3}). Then $f_{5}$
is decreasing and convex on $\left( -1/2,\infty \right) $.
\end{proposition}

\begin{proof}
Differentiation gives%
\begin{eqnarray*}
f_{5}^{\prime }(x) &=&\psi \left( x+1\right) -\ln \left( x+\frac{1}{2}%
\right) -\frac{1517}{44\,640\left( x+1/2\right) ^{2}} \\
&&+\frac{343}{44\,640\left( x^{2}+x+111/196\right) }-\frac{343}{22\,320}%
\frac{\left( x+1/2\right) ^{2}}{\left( x^{2}+x+111/196\right) ^{2}},
\end{eqnarray*}%
\begin{eqnarray*}
f_{5}^{\prime \prime }(x) &=&\psi ^{\prime }\left( x+1\right) -\frac{1}{x+1/2%
}+\frac{1517}{22\,320\left( x+1/2\right) ^{3}} \\
&&-\frac{343}{7440}\frac{x+1/2}{\left( x^{2}+x+111/196\right) ^{2}}+\frac{343%
}{5580}\frac{\left( x+1/2\right) ^{3}}{\left( x^{2}+x+111/196\right) ^{3}}.
\end{eqnarray*}%
Denote by $x+1/2=t$ and make use of recursive relation (\ref{psi-rel.}) yield%
\begin{eqnarray*}
&&f_{5}^{\prime \prime }(t+\frac{1}{2})-f_{5}^{\prime \prime }(t-\frac{1}{2})
\\
&=&-\tfrac{1}{\left( t+1/2\right) ^{2}}-\tfrac{1}{\left( t+1\right) }+\tfrac{%
1517}{22\,320\left( t+1\right) ^{3}}-\tfrac{343}{7440}\tfrac{t+1}{\left(
\left( t+1\right) ^{2}+31/98\right) ^{2}}+\tfrac{343}{5580}\tfrac{\left(
t+1\right) ^{3}}{\left( \left( t+1\right) ^{2}+31/98\right) ^{3}} \\
&&-\left( -\tfrac{1}{t}+\tfrac{1517}{22\,320t^{3}}-\tfrac{343}{7440}\tfrac{t%
}{\left( t^{2}+31/98\right) ^{2}}+\tfrac{343}{5580}\tfrac{t^{3}}{\left(
t^{2}+31/98\right) ^{3}}\right) \\
&=&-\frac{f_{51}\left( t\right) }{80\left( t+1/2\right) ^{2}t^{3}\left(
t+1\right) ^{3}\left( t^{2}+2t+129/98\right) ^{3}\left( t^{2}+31/98\right)
^{3}},
\end{eqnarray*}%
where 
\begin{eqnarray*}
f_{51}\left( t\right) &=&\tfrac{10\,981}{784}t^{10}+\tfrac{54\,905}{784}%
t^{9}+\tfrac{21\,028\,039}{134\,456}t^{8}+\tfrac{27\,614\,911}{134\,456}%
t^{7}+\tfrac{294\,820\,517}{1647\,086}t^{6}+\tfrac{739\,744\,471}{6588\,344}%
t^{5}+ \\
&&\tfrac{138\,266\,105\,451}{2582\,630\,848}t^{4}+\tfrac{25\,165\,604\,049}{%
1291\,315\,424}t^{3}+\tfrac{2726\,271\,884\,261}{506\,195\,646\,208}t^{2}+%
\tfrac{574\,150\,150\,569}{506\,195\,646\,208}t+\tfrac{347\,724\,739\,077}{%
3543\,369\,523\,456} \\
&>&0\text{ for }t=x+1/2>0\text{.}
\end{eqnarray*}%
This implies that $f_{5}^{\prime \prime }(t+\frac{1}{2})-f_{5}^{\prime
\prime }(t-\frac{1}{2})<0$, that is, $f_{5}^{\prime \prime
}(x+1)-f_{5}^{\prime \prime }(x)<0$, and so%
\begin{equation*}
f_{5}^{\prime \prime }(x)>f_{5}^{\prime \prime }(x+1)>f_{5}^{\prime \prime
}(x+2)>...>f_{5}^{\prime \prime }(\infty )=0.
\end{equation*}%
It reveals that shows $f_{5}$ is convex on $\left( -1/2,\infty \right) $,
and therefore, $f_{5}^{\prime }(x)<\lim_{x\rightarrow \infty }f_{5}^{\prime
}(x)=0$, which proves the desired statement.
\end{proof}

Employing the decreasing property of $f_{5}$ on $\left( -1/2,\infty \right) $%
, we obtain

\begin{corollary}
For $x>0$, the double inequality 
\begin{eqnarray*}
&&\sqrt{2\pi }\left( \tfrac{x+1/2}{e}\right) ^{x+1/2}\exp \left( -\tfrac{1517%
}{44\,640}\tfrac{1}{x+1/2}-\tfrac{343}{44\,640}\tfrac{x+1/2}{x^{2}+x+111/196}%
\right) \\
&<&\Gamma (x+1)<e^{2987/39960}\sqrt{2e}\left( \tfrac{x+1/2}{e}\right)
^{x+1/2}\exp \left( -\tfrac{1517}{44\,640}\tfrac{1}{x+1/2}-\tfrac{343}{%
44\,640}\tfrac{x+1/2}{x^{2}+x+111/196}\right)
\end{eqnarray*}%
holds, where $\sqrt{2\pi }\approx 2.5066$ and $e^{2987/39960}\sqrt{2e}%
\approx 2.5126$ are the best constants.

For $n\in \mathbb{N}$, it holds that%
\begin{eqnarray*}
&&\sqrt{2\pi }\left( \tfrac{n+1/2}{e}\right) ^{n+1/2}\exp \left( -\tfrac{1517%
}{44\,640}\tfrac{1}{n+1/2}-\tfrac{343}{44\,640}\tfrac{n+1/2}{n^{2}+n+111/196}%
\right) \\
&<&n!<\frac{2\sqrt{6}}{9}\exp \left( \tfrac{829\,607}{543\,240}\right)
\left( \tfrac{n+1/2}{e}\right) ^{n+1/2}\exp \left( -\tfrac{1517}{44\,640}%
\tfrac{1}{n+1/2}-\tfrac{343}{44\,640}\tfrac{n+1/2}{n^{2}+n+111/196}\right)
\end{eqnarray*}%
with the best constants $\sqrt{2\pi }\approx 2.5066$ and $2\sqrt{6}\exp
\left( \tfrac{829\,607}{543\,240}\right) /9\approx 2.5067$.
\end{corollary}

Lastly, we give an application example of Theorem \ref{MT-p2=p3}.

\begin{example}
let%
\begin{equation*}
M\left( a,b\right) =H_{p,q;r}^{2,1}\left( a,b\right) =\frac{%
pb^{2}+qa^{2}+(1-p-q)ab}{rb+(1-r)a}
\end{equation*}%
and $\theta =0,\sigma =1$ in Theorem \ref{MT-p2=p3}. Then by the same method
previously, we can derive two best arrays%
\begin{eqnarray*}
\left( p_{1},q_{1},r_{1}\right) &=&\left( \frac{129-59\sqrt{3}}{360},\frac{%
129+59\sqrt{3}}{360},\frac{90-29\sqrt{3}}{180}\right) , \\
\left( p_{2},q_{2},r_{2}\right) &=&\left( \frac{129+59\sqrt{3}}{360},\frac{%
129-59\sqrt{3}}{360},\frac{90+29\sqrt{3}}{180}\right) .
\end{eqnarray*}%
Then,%
\begin{eqnarray}
H_{p_{1},q_{1};r_{1}}^{2,1}\left( x,x+1\right) &=&\frac{x^{2}+\frac{180-59%
\sqrt{3}}{180}x+\frac{129-59\sqrt{3}}{360}}{x+\frac{90-29\sqrt{3}}{180}}%
:=M_{1}\left( x,x+1\right) ,  \label{M1} \\
H_{p_{2},q_{2};r_{2}}^{2,1}\left( x,x+1\right) &=&\frac{x^{2}+\frac{180+59%
\sqrt{3}}{180}x+\frac{129+59\sqrt{3}}{360}}{x+\frac{90+29\sqrt{3}}{180}}%
:=M_{2}\left( x,x+1\right)  \label{M2}
\end{eqnarray}%
It is easy to check that $M\left( a,b\right) $ are means of $a$ and $b$ for $%
\left( p,q,r\right) =\left( p_{1},q_{1},r_{1}\right) $ and $\left(
p_{2},q_{2},r_{2}\right) $. Thus, application of Theorem \ref{MT-p2=p3}
implies that both the following two asymptotic formulas%
\begin{equation*}
\ln \Gamma (x+1)\sim \frac{1}{2}\ln 2\pi +\left( x+1/2\right) \ln
M_{i}\left( x,x+1\right) -M_{i}\left( x,x+1\right) \text{, }i=1,2
\end{equation*}%
are valid as $x\rightarrow \infty $. And, we have%
\begin{eqnarray*}
\lim_{x\rightarrow \infty }\tfrac{\ln \Gamma (x+1)-\frac{1}{2}\ln 2\pi
-\left( x+1/2\right) \ln M_{1}\left( x,x+1\right) +M_{1}\left( x,x+1\right) 
}{x^{-4}} &=&-\tfrac{1481\sqrt{3}}{2332\,800}, \\
\lim_{x\rightarrow \infty }\tfrac{\ln \Gamma (x+1)-\frac{1}{2}\ln 2\pi
-\left( x+1/2\right) \ln M_{2}\left( x,x+1\right) +M_{2}\left( x,x+1\right) 
}{x^{-4}} &=&\tfrac{1481\sqrt{3}}{2332\,800}.
\end{eqnarray*}
\end{example}

The above two asymptotic formulas also have well properties.

\begin{proposition}
Let $f_{6},f_{7}$ be defined on $\left( 0,\infty \right) $ by%
\begin{eqnarray*}
f_{6}(x) &=&\ln \Gamma (x+1)-\frac{1}{2}\ln 2\pi -\left( x+1/2\right) \ln
M_{1}\left( x,x+1\right) +M_{1}\left( x,x+1\right) , \\
f_{7}(x) &=&\ln \Gamma (x+1)-\frac{1}{2}\ln 2\pi -\left( x+1/2\right) \ln
M_{2}\left( x,x+1\right) +M_{2}\left( x,x+1\right) ,
\end{eqnarray*}%
where $M_{1}$ and $M_{2}$ are defined\ by (\ref{M1}) and (\ref{M2}),
respectively. Then $f_{6}\ $is increasing and concave on $\left( 0,\infty
\right) $, while $f_{7}$ is decreasing and convex on $\left( 0,\infty
\right) $.
\end{proposition}

\begin{proof}
Differentiation gives%
\begin{eqnarray*}
f_{6}^{\prime }\left( x\right) &=&\psi (x+1)-\ln \frac{x^{2}+\frac{180-59%
\sqrt{3}}{180}x+\frac{129-59\sqrt{3}}{360}}{x+\frac{90-29\sqrt{3}}{180}}-%
\frac{\left( x+\frac{1}{2}\right) \left( 2x+\frac{180-59\sqrt{3}}{180}%
\right) }{x^{2}+\frac{180-59\sqrt{3}}{180}x+\frac{129-59\sqrt{3}}{360}} \\
&&+\frac{x+\frac{1}{2}}{x+\frac{90-29\sqrt{3}}{180}}+\frac{2x+\frac{180-59%
\sqrt{3}}{180}}{x+\frac{90-29\sqrt{3}}{180}}-\frac{x^{2}+\frac{180-59\sqrt{3}%
}{180}x+\frac{129-59\sqrt{3}}{360}}{\left( x+\frac{90-29\sqrt{3}}{180}%
\right) ^{2}},
\end{eqnarray*}%
\begin{eqnarray*}
f_{6}^{\prime \prime }\left( x\right) &=&\psi ^{\prime }(x+1)-\frac{2x+\frac{%
180-59\sqrt{3}}{180}}{x^{2}+\frac{180-59\sqrt{3}}{180}x+\frac{129-59\sqrt{3}%
}{360}}+\frac{1}{x+\frac{90-29\sqrt{3}}{180}} \\
&&+\frac{59\sqrt{3}}{180}\frac{x^{2}+\frac{59-26\sqrt{3}}{59}x+\frac{43}{120}%
-\frac{13\sqrt{3}}{59}}{\left( x^{2}+\frac{180-59\sqrt{3}}{180}x+\frac{129-59%
\sqrt{3}}{360}\right) ^{2}} \\
&&-\frac{7\sqrt{3}}{45}\frac{1}{\left( x+\frac{90-29\sqrt{3}}{180}\right)
^{2}}-\frac{\sqrt{3}}{180}\frac{x-\frac{629\sqrt{3}-90}{180}}{\left( x+\frac{%
90-29\sqrt{3}}{180}\right) ^{3}}.
\end{eqnarray*}%
Employing the recursive relation (\ref{psi-rel.}) and factoring reveal that%
\begin{equation*}
f_{6}^{\prime \prime }\left( x+1\right) -f_{6}^{\prime \prime }\left(
x\right) =\frac{1481\sqrt{3}}{19\,440}\frac{f_{61}\left( x\right) }{%
f_{62}\left( x\right) },
\end{equation*}%
where%
\begin{eqnarray*}
f_{61}\left( x\right) &=&x^{9}+\left( 9-\tfrac{337\,153}{266\,580}\sqrt{3}%
\right) x^{8}+\left( \tfrac{991\,207\,423}{26\,658\,000}-\tfrac{674\,306}{%
66\,645}\sqrt{3}\right) x^{7} \\
&&+\left( \tfrac{2459\,907\,961}{26\,658\,000}-\tfrac{169\,081\,132\,727}{%
4798\,440\,000}\sqrt{3}\right) x^{6}+\left( \tfrac{4335\,292\,090\,469}{%
28\,790\,640\,000}-\tfrac{55\,797\,724\,727}{799\,740\,000}\sqrt{3}\right)
x^{5} \\
&&+\left( \tfrac{956\,621\,902\,709}{5758\,128\,000}-\tfrac{%
148\,442\,768\,304\,491}{1727\,438\,400\,000}\sqrt{3}\right) x^{4} \\
&&+\left( \tfrac{229\,288\,958\,388\,788\,929}{1865\,633\,472\,000\,000}-%
\tfrac{29\,135\,013\,047\,291}{431\,859\,600\,000}\sqrt{3}\right) x^{3} \\
&&+\left( \tfrac{36\,305\,075\,316\,164\,929}{621\,877\,824\,000\,000}-%
\tfrac{55\,416\,459\,045\,055\,111\,861}{1679\,070\,124\,800\,000\,000}\sqrt{%
3}\right) x^{2} \\
&&+\left( \tfrac{179\,958\,708\,278\,174\,628\,611}{11\,193\,800\,832\,000%
\,000\,000}-\tfrac{7731\,435\,289\,282\,423\,861}{839\,535\,062\,400\,000%
\,000}\sqrt{3}\right) x \\
&&+\left( \tfrac{21\,826\,051\,463\,638\,680\,611}{11\,193\,800\,832\,000%
\,000\,000}-\tfrac{5586\,677\,417\,732\,710\,687}{4975\,022\,592\,000\,000%
\,000}\sqrt{3}\right) ,
\end{eqnarray*}%
\begin{eqnarray*}
f_{62}\left( x\right) &=&\left( x+1\right) ^{2}\left( x^{2}+\tfrac{180-59%
\sqrt{3}}{180}x+\tfrac{129-59\sqrt{3}}{360}\right) ^{2}\left( x^{2}+\tfrac{%
540-59\sqrt{3}}{180}x+\tfrac{283-59\sqrt{3}}{120}\right) ^{2} \\
&&\times \left( x+\tfrac{270-29\sqrt{3}}{180}\right) ^{3}\left( x+\tfrac{%
90-29\sqrt{3}}{180}\right) ^{3}.
\end{eqnarray*}%
By direct verifications we see that all coefficients of $f_{61}$ and $f_{62}$
are positive, so $f_{61}\left( x\right) $, $f_{62}\left( x\right) >0$ for $%
x>0$. Therefore, we get $f_{6}^{\prime \prime }\left( x+1\right)
-f_{6}^{\prime \prime }\left( x\right) >0$, which yields%
\begin{equation*}
f_{6}^{\prime \prime }(x)<f_{6}^{\prime \prime }(x+1)<f_{6}^{\prime \prime
}(x+2)<...<f_{6}^{\prime \prime }(\infty )=0.
\end{equation*}%
It shows that $f_{6}$ is concave on $\left( 0,\infty \right) $, and
therefore, $f_{6}^{\prime }(x)>\lim_{x\rightarrow \infty }f_{6}^{\prime
}(x)=0$, which proves the monotonicity and concavity of $f_{6}$.

In the same way, we can prove the monotonicity and convexity of $f_{7}$ on $%
\left( 0,\infty \right) $, whose details are omitted.
\end{proof}

As direct consequences of previous proposition, we have

\begin{corollary}
For $x>0$, the double inequality 
\begin{eqnarray*}
&&\delta _{0}\sqrt{2\pi }\left( \tfrac{x^{2}+\frac{180-59\sqrt{3}}{180}x+%
\frac{129-59\sqrt{3}}{360}}{x+\frac{90-29\sqrt{3}}{180}}\right) ^{x+1/2}\exp
\left( -\tfrac{x^{2}+\frac{180-59\sqrt{3}}{180}x+\frac{129-59\sqrt{3}}{360}}{%
x+\frac{90-29\sqrt{3}}{180}}\right) \\
&<&\Gamma (x+1)<\sqrt{2\pi }\left( \tfrac{x^{2}+\frac{180-59\sqrt{3}}{180}x+%
\frac{129-59\sqrt{3}}{360}}{x+\frac{90-29\sqrt{3}}{180}}\right) ^{x+1/2}\exp
\left( -\tfrac{x^{2}+\frac{180-59\sqrt{3}}{180}x+\frac{129-59\sqrt{3}}{360}}{%
x+\frac{90-29\sqrt{3}}{180}}\right)
\end{eqnarray*}%
holds, where $\delta _{0}=\exp f_{6}\left( 0\right) \approx 0.96259$ and $1$
are the best constants.

For $n\in \mathbb{N}$, it holds that%
\begin{eqnarray*}
&&\delta _{2}\sqrt{2\pi }\left( \tfrac{n^{2}+\frac{180-59\sqrt{3}}{180}n+%
\frac{129-59\sqrt{3}}{360}}{n+\frac{90-29\sqrt{3}}{180}}\right) ^{n+1/2}\exp
\left( -\tfrac{n^{2}+\frac{180-59\sqrt{3}}{180}n+\frac{129-59\sqrt{3}}{360}}{%
n+\frac{90-29\sqrt{3}}{180}}\right) \\
&<&n!<\sqrt{2\pi }\left( \tfrac{n^{2}+\frac{180-59\sqrt{3}}{180}n+\frac{%
129-59\sqrt{3}}{360}}{n+\frac{90-29\sqrt{3}}{180}}\right) ^{n+1/2}\exp
\left( -\tfrac{n^{2}+\frac{180-59\sqrt{3}}{180}n+\frac{129-59\sqrt{3}}{360}}{%
n+\frac{90-29\sqrt{3}}{180}}\right)
\end{eqnarray*}%
with the best constants $\delta _{1}=\exp f_{6}\left( 1\right) \approx
0.99965$ and $1$.
\end{corollary}

\begin{corollary}
For $x>0$, the double inequality 
\begin{eqnarray*}
&&\sqrt{2\pi }\left( \tfrac{x^{2}+\frac{180+59\sqrt{3}}{180}x+\frac{129+59%
\sqrt{3}}{360}}{x+\frac{90+29\sqrt{3}}{180}}\right) ^{x+1/2}\exp \left( -%
\tfrac{x^{2}+\frac{180+59\sqrt{3}}{180}x+\frac{129+59\sqrt{3}}{360}}{x+\frac{%
90+29\sqrt{3}}{180}}\right) \\
&<&\Gamma (x+1)<\tau _{0}\sqrt{2\pi }\left( \tfrac{x^{2}+\frac{180+59\sqrt{3}%
}{180}x+\frac{129+59\sqrt{3}}{360}}{x+\frac{90+29\sqrt{3}}{180}}\right)
^{x+1/2}\exp \left( -\tfrac{x^{2}+\frac{180+59\sqrt{3}}{180}x+\frac{129+59%
\sqrt{3}}{360}}{x+\frac{90+29\sqrt{3}}{180}}\right)
\end{eqnarray*}%
holds, where $\tau _{0}=\exp f_{7}\left( 0\right) \approx 1.0020$ and $1$
are the best constants.

For $n\in \mathbb{N}$, it holds that%
\begin{eqnarray*}
&&\sqrt{2\pi }\left( \tfrac{n^{2}+\frac{180+59\sqrt{3}}{180}n+\frac{129+59%
\sqrt{3}}{360}}{n+\frac{90+29\sqrt{3}}{180}}\right) ^{n+1/2}\exp \left( -%
\tfrac{n^{2}+\frac{180+59\sqrt{3}}{180}n+\frac{129+59\sqrt{3}}{360}}{n+\frac{%
90+29\sqrt{3}}{180}}\right) \\
&<&n!<\tau _{1}\sqrt{2\pi }\left( \tfrac{n^{2}+\frac{180+59\sqrt{3}}{180}n+%
\frac{129+59\sqrt{3}}{360}}{n+\frac{90+29\sqrt{3}}{180}}\right) ^{n+1/2}\exp
\left( -\tfrac{n^{2}+\frac{180+59\sqrt{3}}{180}n+\frac{129+59\sqrt{3}}{360}}{%
n+\frac{90+29\sqrt{3}}{180}}\right)
\end{eqnarray*}%
with the best constants $\delta _{1}=\exp f_{7}\left( 1\right) \approx
1.0001 $ and $1$.
\end{corollary}

\section{Open problems}

Inspired by Examples \ref{E-M3,2}--\ref{E-N4,3}, we propose the following
problems.

\begin{problem}
Let $S_{p_{k};q_{k}}^{n,n-1}\left( a,b\right) $ be defined by (\ref{S^n,n-1}%
). Finding $p_{k}$ and $q_{k}$ such that the asymptotic formula for the
gamma function%
\begin{equation*}
\ln \Gamma (x+1)\sim \frac{1}{2}\ln 2\pi +\left( x+\frac{1}{2}\right) \ln
S_{p_{k};q_{k}}^{n,n-1}\left( x,x+1\right) -\left( x+\frac{1}{2}\right)
:=F_{1}\left( x\right)
\end{equation*}%
holds as $x\rightarrow \infty $ with%
\begin{equation*}
\lim_{x\rightarrow \infty }\frac{\ln \Gamma (x+1)-F_{1}\left( x\right) }{%
x^{-2n+1}}=c_{1}\neq 0,\pm \infty .
\end{equation*}
\end{problem}

\begin{problem}
Let $S_{p_{k};q_{k}}^{n,n-1}\left( a,b\right) $ be defined by (\ref{S^n,n-1}%
). Finding $p_{k}$ and $q_{k}$ such that the asymptotic formula for the
gamma function%
\begin{equation*}
\ln \Gamma (x+1)\sim \frac{1}{2}\ln 2\pi +\left( x+\frac{1}{2}\right) \ln
\left( x+\frac{1}{2}\right) -S_{p_{k};q_{k}}^{n,n-1}\left( x,x+1\right)
:=F_{2}\left( x\right)
\end{equation*}%
holds as $x\rightarrow \infty $ with%
\begin{equation*}
\lim_{x\rightarrow \infty }\frac{\ln \Gamma (x+1)-F_{2}\left( x\right) }{%
x^{-2n+1}}=c_{2}\neq 0,\pm \infty .
\end{equation*}
\end{problem}

\begin{problem}
Let $H_{p_{k};q_{k}}^{n,n-1}\left( a,b\right) $ be defined by (\ref{H^n,n-1}%
). Finding $p_{k}$ and $q_{k}$ such that the asymptotic formula for the
gamma function%
\begin{equation*}
\ln \Gamma (x+1)\sim \frac{1}{2}\ln 2\pi +\left( x+\frac{1}{2}\right) \ln
H_{p_{k};q_{k}}^{n,n-1}\left( x,x+1\right) -H_{p_{k};q_{k}}^{n,n-1}\left(
x,x+1\right) :=F_{3}\left( x\right)
\end{equation*}%
holds as $x\rightarrow \infty $ with%
\begin{equation*}
\lim_{x\rightarrow \infty }\frac{\ln \Gamma (x+1)-F_{1}\left( x\right) }{%
x^{-2n}}=c_{3}\neq 0,\pm \infty .
\end{equation*}
\end{problem}

\end{document}